\documentclass[a4paper,11 pt]{amsart}
\usepackage{stmaryrd}
\usepackage{graphicx}
\usepackage{amssymb}
\usepackage{amsmath}
\usepackage{amsthm}
\usepackage{amscd}
\usepackage[all, knot]{xy}
\xyoption{arc}

\makeatother \makeatletter

\newtheorem{thm}{Theorem}[section]

\newtheorem{lem}[thm]{Lemma}

\theoremstyle{definition}

\theoremstyle{remark}

\numberwithin{equation}{section}

\newcommand{\Rmnum}[1]{\expandafter\@slowromancap\romannumeral #1@}

\def\gr{\operatorname{gr}}

\def \C{\mathcal{C}}
\def\D{\Delta}

\def\v{\varepsilon}

\begin{document}
\title[Hopf Structures on Minimal Hopf Quivers]{Hopf Structures on Minimal Hopf Quivers}
\author{Hua-Lin Huang}
\address{School of Mathematics, Shandong University, Jinan 250100, China} \email{hualin@sdu.edu.cn}
\author{Yu Ye}
\address{Department of Mathematics, University of Science and Technology of China,
 Hefei 230026, China}
\email{yeyu@ustc.edu.cn}
\author{Qing Zhao}
\address{Department of Mathematics, University of Science and Technology of China,
 Hefei 230026, China}
\email{qzhao@mail.ustc.edu.cn}
\maketitle
\date{}

\begin{abstract}
In this paper we investigate pointed Hopf algebras via quiver
methods. We classify all possible Hopf structures arising from
minimal Hopf quivers, namely basic cycles and the linear chain. This
provides full local structure information for general pointed Hopf
algebras.

\vskip 10pt

\noindent{\bf Keywords} \ \ Hopf algebra, Hopf
quiver  \\
\noindent{\bf 2000 MR Subject Classification} \ \ 16W30, 16W35,
16S80
\end{abstract}

\section{Introduction}

Quivers are oriented diagrams consisting of vertices and arrows. Due
to a well-known theorem of Gabriel \cite{gab}, elementary
associative algebras over fields can be presented by path algebras
of quivers modulo admissible ideals in some unique manner and their
representations given by representations of the corresponding bound
quivers. This makes the abstract algebras and their representation
theory visible and plays a central role in the modern representation
theory of associative algebras.

After Gabriel, quiver theory has been established for some other
algebraic structures, in particular for coalgebras and Hopf
algebras. A coalgebra over a field is said to be pointed if its
simple subcoalgebras are one-dimensional, or equivalently, its
simple comodules are one-dimensional. In \cite{cm}, Chin and
Montgomery gave a Gabriel type theorem for pointed coalgebras.
Cibils and Rosso introduced the notion of Hopf quivers \cite{cr2},
which are determined by groups with ramification data, and observed
that the path coalgebra of a quiver admits a graded Hopf structure
if and only if the quiver is a Hopf quiver. A Hopf algebra is called
pointed if its underlying coalgebra is so. Van Oystaeyen and Zhang
established a Gabriel type theorem for graded pointed Hopf algebras
\cite{voz}.

The aforementioned results motivate our project on the study of
general pointed Hopf algebras by taking advantage of quiver methods.
The quiver setting gives a visible frame to classification problem
and representation theory and other respects. The project of quiver
approaches to pointed Hopf algebras consists mainly of the following
aspects:
\begin{enumerate}
\item Classify graded Hopf structures on Hopf quivers. This amounts to a complete classification of graded pointed
      Hopf algebras.
\item Carry out a proper deformation process for graded Hopf structures to get general pointed Hopf algebras.
\item Investigate the (co)representation theory as well as other aspects of the obtained Hopf algebras
      with a help of their quiver presentation.
\end{enumerate}

This paper is conceived as the first step of the project. We
classify Hopf algebra structures on minimal Hopf quivers, namely
basic cycles and the linear train. Intuitively, it is easy to
observe that a general Hopf quiver are compatible combination of
these minimal Hopf quivers. Moreover, on a given Hopf quiver there
is a clear relation between the sub-Hopf algebras and sub-Hopf
quivers, as can be seen from Cibils and Rosso's description of
graded Hopf structures on Hopf quivers \cite{cr2}. Therefore, our
result provides complete basic structure ingredients for general
pointed Hopf algebras.

Most of the Hopf structures arising from minimal Hopf quivers are by
no means novel. They appeared sporadically in various references of
Hopf algebras (see \cite{taft,as1,rad} etc.) and quantum groups (see
\cite{lusztig,dck} etc). Quiver methods provide these Hopf algebras
in a unified setting. Aside from quiver techniques, our arguments
also rely on Bergman's Diamond Lemma \cite{b} which helps to present
the Hopf algebras by generators with relations. This is useful in
carrying out the preferred deformation procedure in the sense of
Gerstenhaber and Schack \cite{g-s}.

In principle general pointed Hopf algebras can be obtained by
compatible gluing of the minimal ones we obtain. By consulting the
classical Lie theory, viewing the minimal Hopf structures as quantum
version of the Borel part of $sl_2,$ it is natural to expect that
the gluing problem should turn out to be a rich theory and (quantum
versions of) Cartan matrices, Weyl groups and root systems should
get involved. We leave this for future work. In the situation of
finite-dimensional pointed Hopf algebras with abelian group-likes,
Andruskiewitsch and Schneider have made substantial progress in
classification problem \cite{as2} by different method initiated in
\cite{as1}.

The paper is organized as follows. In Section 2 we review some
necessary facts about Hopf quivers and pointed Hopf algebras. In
Sections 3 and 4, the explicit classifications of Hopf structures on
minimal Hopf quivers are given. Section 5 is devoted to some
applications of the classification results.

Throughout the paper, we work over an algebraically closed field of
characteristic zero. The readers are referred to \cite{grk,ass} for
general knowledge of quivers and representations, and to \cite{sw,
mont1} for that of coalgebras and Hopf algebras.

\section{Quiver Approaches to Pointed Hopf Algebras}

For the convenience of the reader, in this section we recall some
basic notions and facts in \cite{cr2, voz}. Note that there is a
dual approach to elementary Hopf algebras via quivers, see
\cite{c1,g,cr1,gs} for related works.

\subsection{}
A quiver is a quadruple $Q=(Q_0,Q_1,s,t),$ where $Q_0$ is the set of
vertices, $Q_1$ is the set of arrows, and $s,t:\ Q_1 \longrightarrow
Q_0$ are two maps assigning respectively the source and the target
for each arrow. A path of length $l \ge 1$ in the quiver $Q$ is a
finitely ordered sequence of $l$ arrows $a_l \cdots a_1$ such that
$s(a_{i+1})=t(a_i)$ for $1 \le i \le l-1.$ By convention a vertex is
said to be a trivial path of length $0.$

The path coalgebra $kQ$ is the $k$-space spanned by the paths of $Q$
with counit and comultiplication maps defined by $\v(g)=1, \ \D(g)=g
\otimes g$ for each $g \in Q_0,$ and for each nontrivial path $p=a_n
\cdots a_1, \ \v(p)=0,$
\begin{equation}
\D(a_n \cdots a_1)=p \otimes s(a_1) + \sum_{i=1}^{n-1}a_n \cdots
a_{i+1} \otimes a_i \cdots a_1 \nonumber + t(a_n) \otimes p \ .
\end{equation}
The length of paths gives a natural gradation to the path coalgebra.
Let $Q_n$ denote the set of paths of length $n$ in $Q,$ then
$kQ=\oplus_{n \ge 0} kQ_n$ and $\D(kQ_n) \subseteq
\oplus_{n=i+j}kQ_i \otimes kQ_j.$ Clearly $kQ$ is pointed with the
set of group-likes $G(kQ)=Q_0,$ and has the following coradical
filtration $$ kQ_0 \subseteq kQ_0 \oplus kQ_1 \subseteq kQ_0 \oplus
kQ_1 \oplus kQ_2 \subseteq \cdots.$$ Hence $kQ$ is coradically
graded.

According to \cite{cr2}, a quiver $Q$ is said to be a Hopf quiver if
the corresponding path coalgebra $kQ$ admits a graded Hopf algebra
structure. Hopf quivers can be determined by ramification data of
groups. Let $G$ be a group, $\C$ the set of conjugacy classes. A
ramification datum $R$ of the group $G$ is a formal sum $\sum_{C \in
\C}R_CC$ of conjugacy classes with coefficients in
$\mathbb{N}=\{0,1,2,\cdots\}.$ The corresponding Hopf quiver
$Q=Q(G,R)$ is defined as follows: the set of vertices $Q_0$ is $G,$
and for each $x \in G$ and $c \in C,$ there are $R_C$ arrows going
from $x$ to $cx.$

A Hopf quiver $Q=Q(G,R)$ is connected if and only if the union of
the conjugacy classes with non-zero coefficients in $R$ generates
$G.$ We denote the unit element of $G$ by $e.$ If $R_{ \{e\} } \ne
0,$ then there are $R_{ \{e\} }$-loops attached to each vertex; if
the order of elements in a conjugacy class $C \ne \{e\}$ is $n$ and
$R_C \ne 0,$ then corresponding to these data in $Q$ there is a
sub-quiver $(n,R_C)$-cycle (called basic $n$-cycle if $R_C=1$),
i.e., the quiver having $n$ vertices, indexed by the set of integers
modulo $n,$ and $R_C$ arrows going from $\overline{i}$ to
$\overline{i+1}$ for each $i;$ if the order of elements in a
conjugacy class $C$ is $\infty,$ then in $Q$ there is a sub-quiver
$R_C$-chain (called linear chain if $R_C=1$), i.e., a quiver having
set of vertices indexed by the set of integral numbers, and $R_C$
arrows going from $j$ to $j+1$ for each $j.$ Therefore, basic cycles
(including loop as 1-cycle) and the linear train are basic building
brocks of general Hopf quivers.

Due to cibils and Rosso \cite{cr2}, for a given Hopf quiver $Q,$ the
set of graded Hopf structures on $kQ$ is in one-to-one
correspondence with the set of $kQ_0$-Hopf bimodule structures on
$kQ_1.$ The graded Hopf structures are obtained from Hopf bimodules
via quantum shuffle product \cite{rosso}. The graded Hopf structures
can be restricted to sub-Hopf quivers, hence for the very local
sub-Hopf structures it suffices to consider those arising from
minimal Hopf quivers.

\subsection{}

Let $H$ be a pointed Hopf algebra. Denote its coradical filtration
by $\{ H_n \}_{n=0}^{\infty}.$ Define \[\operatorname{gr}(H)=H_0
\oplus H_1/H_0 \oplus H_2/H_1 \oplus \cdots \cdots \] as the
corresponding (coradically) graded coalgebra. Then
$\operatorname{gr}(H)$ inherits from $H$ a coradically graded Hopf
algebra structure (see e.g. \cite{mont1}). Note that any generating
set of $\operatorname{gr}(H)$ (as an algebra) can be lifted to one
for $H.$ This useful fact can be verified easily by induction.

\begin{lem}
Assume that $\mathcal{G} \subset \operatorname{gr}(H)$ is a
generating set and $\widetilde{\mathcal{G}} \subset H$ an arbitrary
set of its representatives. Then $\widetilde{\mathcal{G}}$ generates
$H.$
\end{lem}

The procedure from $H$ to $\gr H$ is called degeneration. The
converse procedure is called deformation. According to Gerstenhaber
and Schack \cite{g-s}, a coalgebra-preserving deformation is called
preferred. If we want to classify all the Hopf structures on the
whole path coalgebra of a Hopf quiver, or the bialgebras of type one
\cite{n}, then we only need to carry out preferred deformation
procedure.

According to Van Oystaeyen and Zhang \cite{voz}, if $H$ is a
coradically graded pointed Hopf algebra, then there exists a unique
Hopf quiver $Q(H)$ such that $H$ can be realized as a large sub-Hopf
algebra of a graded Hopf structure on the path coalgebra $kQ(H).$
Here by ``large" we mean $H$ contains the subspace $kQ(H)_0 \oplus
kQ(H)_1.$ This Gabriel type theorem allows us to classify pointed
Hopf algebras exhaustively in the quiver setting. The combinatorial
structure of Hopf quivers implies clearly a Cartier-Gabriel
decomposition theorem (see e.g. \cite{sw,mont1}) for general pointed
Hopf algebras as given by Montgomery \cite{mont2}. It suffices to
study only Hopf structures on connected Hopf quivers.

\section{Hopf Structures on Basic Cycles}

\subsection{}
Let $G=<g \ | \ g^n=1>$ be a cyclic group of order $n$ and let
$\mathcal{Z}$ denote the Hopf quiver $Q(G,g).$ The quiver
$\mathcal{Z}$ is a basic $n$-cycle and this is the only possible way
it is realized as a Hopf quiver. If $n=1,$ then $\mathcal{Z}$ is the
one-loop quiver, that is, consisting of one vertex and one loop. It
is easy to see that such a quiver provides only the familiar divided
power Hopf algebra in one variable, which is isomorphic to the
polynomial algebra in one variable.

From now on we assume $n > 1$ and fix a basic $n$-cycle
$\mathcal{Z}.$ For each integer $i$ modulo $n,$ let $a_i$ denote the
arrow $g^i \longrightarrow g^{i+1}.$ Let $p_i^l$ denote the path
$a_{i+l-1} \cdots a_{i+1}a_i$ of length $l.$ Then $\{p_i^l \ | \ 0
\le i \le n-1, \ l \ge 0 \}$ is a basis of $k\mathcal{Z}.$

Before moving on, we fix some notations of Gaussian binomials. For
any $q \in k,$ integers $l,m \ge 0,$ let
\[
l_q=1+ q + \cdots +q^{l-1}, \ \ l!_q=1_q \cdots l_q, \ \
\binom{l+m}{l}_q=\frac{(l+m)!_q}{l!_qm!_q}.
\]
When $1\ne q\in k$ is an $n$-th root of unity of order $d,$
\[
{l+m \choose l}_q = 0 \ \ \text{if and only if} \ \
\big[\frac{l+m}{d} \big]-\big[ \frac{m}{d}\big]-\big[ \frac{l}{d}
\big]>0,
\]
where $[x]$ means the integer part of $x$.

We need the following fact about automorphisms of the path coalgebra
$k\mathcal{Z}$ in our later argument.

\begin{lem}
Let $d>1$ be an integer and $k\mathcal{Z}[d]$ the subcoalgebra
$\bigoplus_{i=0}^d k\mathcal{Z}_i.$ For any $\lambda \in k$ the
following linear map
\begin{alignat*}{2}
f^d_{\lambda}(0): \ & k\mathcal{Z}[d] \longrightarrow k\mathcal{Z}[d]\\
& p_i^l \mapsto p_i^l, \quad \forall \ i, \ 0 \le l \le d-1, \\
& p_0^d \mapsto p_0^d+\lambda(1-g^d),\\
& p_i^d \mapsto p_i^d, \quad \ 1 \le i \le n-1.
\end{alignat*}
defines a coalgebra automorphism of $k\mathcal{Z}[d].$ There exists
a coalgebra automorphism $F^d_\lambda(0): \ k\mathcal{Z}
\longrightarrow k\mathcal{Z}$ such that its restriction to
$k\mathcal{Z}[d]$ is $f_\lambda^d(0).$ Similar map
$f^d_{\lambda}(j)$ can be defined and be extended to
$F^d_\lambda(j)$ for any $j.$ Moreover, any automorphism of the
subcoalgebra $k\mathcal{Z}[d]$ with restriction to
$k\mathcal{Z}[d-1]$ being identity is a finite composition of some
$f_\lambda^d(j)$'s. Therefore all such automorphisms are extendable
to automorphisms of the path coalgebra $k\mathcal{Z}.$
\end{lem}

\begin{proof}
The claim that $f^d_\lambda(0)$ is a coalgebra automorphism is
obvious. For the second claim, define the map $F^d_\lambda(0)$ as
follows:
\begin{eqnarray*}
F^d_{\lambda}(0): && k\mathcal{Z}\longrightarrow k\mathcal{Z} \\
&& p_i^l \mapsto p_i^l, \quad \forall \ i, \ 0 \le l \le d-1, \\
&& p_0^d \mapsto p_0^d+\lambda(1-g^d), \\
&& p_i^d \mapsto p_i^d, \quad \ 1 \le i \le n-1, \ \mathrm{and \ for}\ l > d,\\
&& p_i^l \mapsto \left\{
                   \begin{array}{ll}
                     p_0^l-\lambda p_d^{l-d}, & \hbox{$i=0,\ l \ne d \ (\mathrm{mod}\ n)$;} \\
                     p_0^l+\lambda p_0^{l-d}-\lambda p_d^{l-d}, & \hbox{$i=0,\ l = d \ (\mathrm{mod}\ n)$;} \\
                     p_i^l+\lambda p_i^{l-d}, & \hbox{$1 \le i \le n-1, \ i+l = d \ (\mathrm{mod}\ n)$;} \\
                     p_i^l, & \hbox{$1 \le i \le n-1, \ i+l \ne d \ (\mathrm{mod}\ n)$.}
                   \end{array}
                 \right.
\end{eqnarray*}
It is straightforward (but a bit tedious) to verify that
$F_\lambda^d(0)$ is the desired coalgebra automorphism of
$k\mathcal{Z}.$ The rest claims are easy.
\end{proof}

\subsection{}
First we recall the graded Hopf structures on $k\mathcal{Z}.$ By
\cite{cr2}, they are in one-to-one correspondence with the
$kG$-module structures on $ka_0,$ and in turn with the set of $n$-th
root of unity. For each $q \in k$ with $q^n=1,$ let $g.a_0=qa_0$
define a $kG$-module. The corresponding $kG$-Hopf bimodule is $kG
\otimes_{kG} ka_0 \otimes kG = ka_0 \otimes kG.$ We identify
$a_i=a_0 \otimes g^i.$ This is how we view $k\mathcal{Z}_1$ as a
$kG$-Hopf bimodule. The following path multiplication formula
\begin{equation}
p_i^l \cdot p_j^m= q^{im} {{l+m}\choose l} _q p_{i+j}^{l+m} \ .
\end{equation}
was given in \cite{cr2} by induction. In particular,
\begin{equation}
g \cdot p_i^l=q^lp_{i+1}^l, \ \ p_i^l \cdot g=p_{i+1}^l, \ \
a_0^l=l_q!p_0^l \ .
\end{equation}
For each $q,$ the corresponding graded Hopf algebra is denoted by
$k\mathcal{Z}(q).$

We consider in the following lemma the algebraic side of
$k\mathcal{Z}(q).$ The facts are our starting point of the preferred
deformation process.

\begin{lem}
As an algebra, $k\mathcal{Z}(q)$ can be presented by generators with
relations as follows:
\begin{enumerate}
  \item when $q=1,$ generators: $g, \ a_0.$ relations: $g^n=1,\ ga_0=a_0g.$
  \item when $\operatorname{ord}(q)=d>1,$ generators: $g, \ a_0, \ p_0^d.$
  relations: $g^n=1, \ ga_0=qa_0g, \ a_0^d=0, \ a_0p_0^d=p_0^da_0, \ gp_0^d=p_0^dg.$
\end{enumerate}
\end{lem}

\begin{proof}
The claim about the generators and the relations they satisfy is
direct consequence of (3.1) and (3.2). In particular, for the case
$\operatorname{ord}(q)=d>1,$ we have
\begin{equation}
(p_0^d)^l=p_0^{dl}, \ \ p_0^{dl} a_0^j=j!_qp_0^{j+dl} \ .
\end{equation}
It suffices to prove conversely the relations are enough to define
$k\mathcal{Z}(q).$

Let $H(q)$ denote the algebra defined in the lemma. To avoid
confusion, we use new notations for the generators: change $g$ to
$h,$ $a_0$ to $a,$ and $p_0^d$ to $p.$ The relations are obtained by
substituting the old notations by the new ones.

For the case $q=1,$ by the well-known diamond lemma \cite{b}
$\{a^lh^i \ | \ 0 \le i \le n-1, \ l \ge 0 \}$ is a basis of $H(1).$
Now define a linear map $f: \ H(1) \longrightarrow k\mathcal{Z}(1),
\ \ a^lh^i \longmapsto l!p_i^l \ .$ Apparently this is a linear
isomorphism. It remains to check that it respects the
multiplication. This is again direct consequence of (3.1) and (3.2):
\[
f((a^lh^i)(a^mh^j))=(l+m)!p_{i+j}^{l+m}
=(l!p_i^l)(m!p_j^m)=f(a^lh^i)f(a^mh^j) \ .
\]

For the case $\operatorname{ord}(q)=d>1,$ the set $\{p^la^jh^i \ | \
0 \le i \le n-1, \ 0 \le j \le d-1, \ l \ge 0 \}$ is a basis of
$H(q),$ again by the diamond lemma. Define a linear map
\[
h: \ H(q) \longrightarrow k\mathcal{Z}(q), \ \ p^la^jh^i \longmapsto
j!_q p_i^{j+dl} \ .
\]
Similarly one can verify by direct computation with a help of (3.1)
and (3.3) that this is an algebra isomorphism:
\begin{eqnarray*}
f((p^la^jh^i) (p^{l'}a^{j'}h^{i'}) &=&
q^{ij'}(j+j')!_qp_{i+i'}^{j+j'+d(l+l')} \\  &=&
(j!_qp_i^{j+dl})(j'!_qp_{i'}^{j'+dl'}) \\ &=& f(p^la^jh^i)
f(p^{l'}a^{j'}h^{i'}) \ .
\end{eqnarray*}
\end{proof}

\subsection{}
Now we are ready to state the the main result of this section. We
classify all the (non-graded) Hopf structures on the path coalgebra
$k\mathcal{Z}.$

\begin{thm} Let $H$ be a Hopf structure on $k\mathcal{Z}$
with $\operatorname{gr}H \cong k\mathcal{Z}(q).$ Then as algebra, it
can be presented by generators and relations as follows:
\begin{enumerate}
  \item if $q=1,$ generators: $g, \ a_0.$ relations: $g^n=1, \
  ga_0=a_0g.$ In particular, the Hopf algebra $H$ is isomorphic to
  $k\mathcal{Z}(1).$
  \item if $\operatorname{ord}(q)=n,$ generators: $g, \ a_0, \ p_0^n.$ relations: $g^n=1,
  \ a_0^n=0, \ ga_0=qa_0g, \ gp_0^n=p_0^ng, \ a_0p_0^n-p_0^na_0=\lambda a_0$ with some
  $\lambda \in k.$
  \item if $1<\operatorname{ord}(q)=d<n,\ n \ne 2d,$ generators: $g, \ a_0, \ p_0^d.$ relations: $g^n=1,
  \ a_0^d=0, \ ga_0=qa_0g, \ gp_0^d=p_0^dg, \ a_0p_0^d-p_0^da_0=0.$
  In other words, the Hopf algebra $H$ is isomorphic to $k\mathcal{Z}(q).$
  \item if $n=2d$ is even and $\operatorname{ord}(q)=d,$ generators: $g, \ a_0, \ p_0^n.$ relations: $g^n=1,\
  a_0^d=\mu(1-g^d),\ ga_0=qa_0g, \ gp_0^d=p_0^dg, \ a_0p_0^d-p_0^da_0= \frac{\mu(1-q)}{(d-1)_q}a_0(1+g^d)$
  with some $\mu \in k.$
\end{enumerate}
\end{thm}

The proof will be separated into several steps. The main idea is to
determine all the possible preferred deformations from the graded
ones, with a help of the quiver.

\subsection{}
If $\operatorname{gr}H \cong k\mathcal{Z}(1),$ then $H$ is generated
by $g$ and $a_0,$ by Lemma 3.2 (1). We only need to give all the
possible relations involved by them. This suffices to consider all
the possible preferred deformations of the graded generating
relations. In this situation, we need to determine the
\[
lower \ terms=ga_0g^{-1}-a_0 \ .
\]
By $\Delta(g \cdot a_0 \cdot
g^{-1})=\Delta(g)\Delta(a_0)\Delta(g^{-1})=g \cdot a_0 \cdot g^{-1}
\otimes 1 +g \otimes g \cdot a_0 \cdot g^{-1} \ ,$ we can conclude
that $g \cdot a_0 \cdot g^{-1} \in \ ^g(k\mathcal{Z})^1,$ hence
\[
g \cdot a_0 \cdot g^{-1}-a_0 = \lambda (1-g)
\]
for some $\lambda \in k.$ The relation $g^n=1$ is stable under
deformation. Note that
\[
a_0=g^n \cdot a_0 \cdot g^{-n}=a_0+n\lambda(1-g) \ ,
\]
This forces $\lambda = 0.$ Therefore there are no non-trivial
preferred deformations for $k\mathcal{Z}(1).$

\subsection{}
If $\operatorname{ord}(q)=n$ and $\operatorname{gr}H \cong
k\mathcal{Z}(q),$ then $H$ is generated by $g, \ a_0, \ p_0^n$ by
Lemma 3.2 (2). As before, we need to determine all the possible
preferred deformations of the graded generating relations in Lemma
3.2 (2).

Firstly, the relation $ga_0=qa_0g$ might be deformed to
\[ga_0g^{-1}=qa_0+\alpha(1-g)\] for some $\alpha \in k.$ Let
$\widetilde{a_0}=a_0-\frac{\alpha}{1-q}(1-g),$ then we have
$g\widetilde{a_0}g^{-1}=q\widetilde{a_0}.$ Set
$\lambda=\frac{\alpha}{1-q}$ and $f_\lambda^1$ as in Lemma 3.1. It
follows that the map $f_\lambda^1$ can be extended to a coalgebra
automorphism $F^1_\lambda$ of $k\mathcal{Z}.$ Now the original Hopf
structure can be transported through $F^1_\lambda$ to one with
$ga_0=qa_0g.$ By iterative application of the lemma, we can have
through coalgebra automorphism (or base change)
\begin{equation}
a_0g^i=a_i,\ a_0^lg^i=l_q!p_i^l,\ i=0,1,\cdots,n-1,\ l=1,\cdots,n-1.
\end{equation}
Note that under such coalgebra automorphisms, the ``new" elements
$g, \ a_0, \ p_0^n$ preserve to be generators of $H$ according to
Lemma 3.2.

Secondly, consider the relation $a_0^n=0.$ In order to see to what
it may be deformed, we should look at $\Delta(a_0^n).$  By the
Gaussian binomial formula (see e.g. \cite{kassel}, Prop. IV.2.2), we
have
\begin{eqnarray*}
\Delta(a_0^n)&=&(\Delta(a_0))^n=(a_0 \otimes 1 + g \otimes a_0)^n \\
&=&\sum_{i=0}^n \binom{n}{i}_q a_0^ig^{n-i} \otimes a_0^{n-i} \\ &=&
a_0^n \otimes 1 +1 \otimes a_0^n \ .
\end{eqnarray*}
This follows that
\begin{equation} a_0^n=0, \end{equation}
since in $k\mathcal{Z}$ there is no loop attached to $1.$

Finally we consider the relations involved $p_0^n.$ Similarly, by
\[
\Delta(gp_0^n-p_0^ng)=(gp_0^n-p_0^ng) \otimes g + g \otimes
(gp_0^n-p_0^ng)
\]
we have $gp_0^n-p_0^ng=0.$ By (3.4) and (3.5) we have the following
equations:
\begin{eqnarray*}
\D(a_0p_0^n)&=&\D(a_0)\D(p_0^n)\\
            &=&(a_0 \otimes 1 +g \otimes a_0)(\sum_{l=0}^np_l^{n-l}
                \otimes p_0^l) \\
            &=&\sum_{l=0}^n a_0p_l^{n-l} \otimes p_0^l +
               \sum_{l=0}^n gp^{n-l}_l \otimes a_0p_0^l \\
            &=&\sum_{l=1}^n p_l^{n+1-l} \otimes p_0^l + a_0p_0^n
               \otimes 1 +g^{n+1} \otimes a_0p_0^n \\
\D(p_0^na_0)&=&\D(p_0^n)\D(a_0)\\
            &=&(\sum_{l=0}^np_l^{n-l}
                \otimes p_0^l)(a_0 \otimes 1 +g \otimes a_0) \\
            &=&\sum_{l=0}^n p_l^{n-l}a_0 \otimes p_0^l +
               \sum_{l=0}^n p^{n-l}_lg \otimes p_0^la_0 \\
            &=&\sum_{l=1}^n p_l^{n+1-l} \otimes p_0^l + p_0^na_0
               \otimes 1 +g^{n+1} \otimes p_0^na_0 \\
\end{eqnarray*}
Let $[a_0, p_0^n]=a_0p_0^n-p_0^na_0^n.$ Then the previous equations
give rise to
\[
\D([a_0, p_0^n])=[a_0, p_0^n] \otimes 1 + g \otimes [a_0, p_0^n].
\]
Now from the structure of the space of $(1,g)$-primitive elements of
$k\mathcal{Z},$ we have \[ [a_0, p_0^n]=\lambda a_0 + \mu(1-g) \]
for some $\lambda,\mu \in k.$ By induction, one gets
\[
a_0^np_0^n=p_0^na_0^n + n \lambda a_0^n + n\mu a_0^{n-1} \ .
\]
With (3.4) and (3.5), this forces $\mu=0.$

\subsection{}
Now assume $1<\operatorname{ord}(q)=d<n$ and $\operatorname{gr}H
\cong k\mathcal{Z}(q),$ then $H$ is generated by $g, \ a_0, \
p_0^d.$ Repeat the argument in Subsection 3.5, we can assume without
loss of generality for $i=0,1,\cdots,n-1,\ l=1,\cdots,d-1$ that
\begin{equation}
ga_0=qa_0g,\ a_0g^i=a_i,\ a_0^lg^i=l_q!p_i^l.
\end{equation}

Consider $\D(a_0^d).$ By Gaussian binomial formula again, we have
\begin{eqnarray*}
\Delta(a_0^d)&=&(\Delta(a_0))^d=(a_0 \otimes 1 + g \otimes a_0)^d \\
&=&\sum_{i=0}^d \binom{d}{i}_q a_0^ig^{d-i} \otimes a_0^{d-i} \\ &=&
a_0^d \otimes 1 +g^d \otimes a_0^d \ .
\end{eqnarray*}
Since in $k\mathcal{Z}$ there is no arrow going from $1$ to $g^d,$
hence it follows that
\begin{equation}
a_0^d=\mu (1-g^d),
\end{equation}
with some $\mu \in k.$

We continue to consider the relations involved $p_0^d.$ By
\[
\Delta(gp_0^d-p_0^dg)=(gp_0^d-p_0^dg) \otimes g + g^{d+1} \otimes
(gp_0^d-p_0^dg)
\]
it follows $gp_0^d-p_0^dg=\nu(g-g^{d+1})$ for some $\nu \in k.$ By
induction we have $g^np_0^d-p_0^dg^n=n\nu(1-g^d).$ Since $g^n=1,$ we
conclude that $\nu=0$ and
\begin{equation}
gp_0^d=p_0^dg.
\end{equation}

Finally we consider $\D([a_0,p_0^d]).$ By (3.6), (3.7) and (3.8) we
have the following:
\begin{eqnarray*}
\D(a_0p_0^d)&=&\D(a_0)\D(p_0^d)\\
            &=&(a_0 \otimes 1 +g \otimes a_0)(\sum_{l=0}^dp_l^{d-l}
                \otimes p_0^l) \\
            &=&\sum_{l=0}^d a_0p_l^{d-l} \otimes p_0^l +
               \sum_{l=0}^d gp^{d-l}_l \otimes a_0p_0^l \\
            &=&\sum_{l=1}^d p_l^{d+1-l} \otimes p_0^l + a_0p_0^d
               \otimes 1 + g^{d+1} \otimes a_0p_0^d \\
            & & + a_0p_1^{d-1} \otimes a_0 + gp_{d-1}^1 \otimes a_0p_0^{d-1}\\
            &=& \sum_{l=1}^d p_l^{d+1-l} \otimes p_0^l + a_0p_0^d
               \otimes 1 + g^{d+1} \otimes a_0p_0^d \\
            & & + \frac{\mu}{(d-1)_q!}(g-g^{d+1}) \otimes a_0 +\frac{q\mu}{(d-1)_q!}p_d^1 \otimes (g-g^{d+1})\\
\D(p_0^da_0)&=&\D(p_0^d)\D(a_0)\\
            &=&(\sum_{l=0}^dp_l^{d-l}
                \otimes p_0^l)(a_0 \otimes 1 +g \otimes a_0) \\
            &=&\sum_{l=0}^d p_l^{d-l}a_0 \otimes p_0^l +
               \sum_{l=0}^d p^{d-l}_lg \otimes p_0^la_0 \\
            &=&\sum_{l=1}^d p_l^{d+1-l} \otimes p_0^l + p_0^da_0
               \otimes 1+ g^{d+1} \otimes p_0^da_0 \\
            & &+ p_1^{d-1}a_0 \otimes a_0 + p_{d-1}^1g \otimes p_0^{d-1}a_0\\
            &=& \sum_{l=1}^d p_l^{d+1-l} \otimes p_0^l + a_0p_0^d
               \otimes 1 + g^{d+1} \otimes a_0p_0^d \\
            & & + \frac{q\mu}{(d-1)_q!}(g-g^{d+1}) \otimes a_0 +\frac{\mu}{(d-1)_q!}p_d^1 \otimes (g-g^{d+1})\\
\end{eqnarray*}
These equations lead to
\begin{eqnarray*}
& &\D\{[a_0,p_0^d]-\frac{\mu(1-q)}{(d-1)_q!}(a_0+p_d^1)\} \\
&=&\{[a_0,p_0^d]-\frac{\mu(1-q)}{(d-1)_q!}(a_0+p_d^1)\} \otimes 1 +
g^{d+1} \otimes
\{[a_0,p_0^d]-\frac{\mu(1-q)}{(d-1)_q!}(a_0+p_d^1)\}\ .
\end{eqnarray*}
It follows as before that
\[
[a_0,p_0^d]=\frac{\mu(1-q)}{(d-1)_q!}a_0(1+g^d)+\lambda(1-g^{d+1})
\]
for some $\lambda \in k.$ Again by induction we have
\[
a_0^dp_0^d=p_0^da_0^d +
d\frac{\mu^2(1-q)}{(d-1)_q!}(1-g^{2d})+d\lambda a_0^{d-1} \ .
\]
So $\mu=\lambda=0$ if $n \ne 2d,$ or $\lambda=0$ otherwise. Now we
can conclude that
\begin{equation}
[a_0,p_0^d]=\left\{
              \begin{array}{ll}
                \frac{\mu(1-q)}{(d-1)_q!}a_0(1+g^d), & \hbox{if n=2d;} \\
                0, & \hbox{otherwise.}
              \end{array}
            \right.
\end{equation}

\subsection{}
So far we have proved that the Hopf structures on $k\mathcal{Z}$
must be generated by $g,\ a_0, \ p_0^d$ (where
$d=\operatorname{ord}(q))$ and satisfy the relations presented in
Theorem 3.3. To complete the proof it suffices to verify that these
relations are enough to define Hopf structures on $k\mathcal{Z}.$ We
only need to prove the cases of $\operatorname{ord}(q)=n$ and
$\operatorname{ord}(q)=n/2,$ since otherwise the Hopf structures are
graded and was done in Lemma 3.2.

The verification is sort of routine, similar to that for graded
case. We only prove the case of $\operatorname{ord}(q)=n.$ The case
of $\operatorname{ord}(q)=n/2$ can be done in a similar manner.
Assume an algebra $\mathcal{C}(q,\lambda)$ is defined by generators
$h, \ a, \ p$ with relations
\[h^n=1, \ a^n=0, \ ha=qah, \ hp=ph, \ ap-pa=\lambda a. \] By the
diamond lemma, the algebra has \[\{ p^ka^jh^i \ | \ 0 \le i,j \le
n-1, \ k \ge 0 \}\] as a basis. Since the Hopf algebra $H$ has a
basis of similar form
\[\{ (p_0^n)^ka_0^jg^i \ | \ 0 \le i,j \le n-1, \ k \ge 0 \},\]
we can define a linear isomorphism $F: \ \mathcal{C}(q,\lambda)
\longrightarrow H$ by sending $p^ka^jh^i$ to $(p_0^n)^ka_0^jg^i.$ It
is an algebra map by direct calculation. This completes the proof.

\subsection{}
We summarize in the following all the Hopf algebra structures living
on $k\mathcal{Z}.$ We denote by $k\mathcal{Z}(n,q,\lambda)$ the Hopf
algebra defined by (2) of Theorem 3.3, and by
$k\mathcal{Z}(\frac{n}{2},q,\mu)$ the Hopf algebra defined by (4) of
Theorem 3.3. The verification of the statements about Hopf algebra
isomorphism is routine, so is omitted.

\begin{thm}
Let $\mathcal{Z}$ be a basic $n$-cycle and $k\mathcal{Z}$ the
associated path coalgebra.
\begin{enumerate}
  \item If $n$ is odd, then the Hopf structures on $k\mathcal{Z}$ are
        given by $k\mathcal{Z}(q)$ (graded) and
        $k\mathcal{Z}(n,q,\lambda)$ (non-graded). We have Hopf
        algebra isomorphism $k\mathcal{Z}(q) \cong k\mathcal{Z}(q')$
        if and only if $q=q';$ $k\mathcal{Z}(n,q,\lambda) \cong
        k\mathcal{Z}(n,q',\lambda')$ if and only if $q=q'$ and there
        exists some $0 \ne \zeta \in k$ such that $\lambda =
        \zeta^n \lambda'.$
  \item If $n$ is even, then the Hopf structures on $k\mathcal{Z}$ are
        given by $k\mathcal{Z}(q)$ (graded), $k\mathcal{Z}(n,q,\lambda)$ and
        $k\mathcal{Z}(\frac{n}{2},q,\mu)$ (non-graded). We have Hopf
        algebra isomorphism $k\mathcal{Z}(\frac{n}{2},q,\mu) \cong
        k\mathcal{Z}(\frac{n}{2},q',\mu')$ if and only if $q=q'$ and there
        exists some $0 \ne \zeta \in k$ such that $\mu =
        \zeta^{\frac{n}{2}} \mu'.$
\end{enumerate}
\end{thm}

\section{Hopf Structures on the Linear Chain}

\subsection{}
Let $G=<g>$ be an infinite cyclic group and let $\mathcal{A}$ denote
the Hopf quiver $Q(G,g).$ Then $\mathcal{A}$ is the linear chain. We
remark that this is the only possible way to view it as a Hopf
quiver. Let $e_i$ denote the arrow $g^i \longrightarrow g^{i+1}$ and
$p_i^l$ the path $e_{i+l-1} \cdots e_i$ of length $l \ge 1,$ for
each $i \in \mathbb{Z}.$ The notation $p_i^0$ is understood as
$e_i.$

Similar to the case of basic cycle, we need the following lemma to
make appropriate base change in later argument. The proof is almost
identical to Lemma 3.1, so we omit the detail.

\begin{lem}
Let $k\mathcal{A}[d]$ be the subcoalgebra $\bigoplus_{i=0}^d
k\mathcal{A}_i.$ For any $\lambda \in k$ the following linear map
\begin{eqnarray*}
f^d_{\lambda}: && k\mathcal{A}[d] \longrightarrow k\mathcal{A}[d] \\
&& p_i^l \mapsto p_i^l, \quad \forall \ i, \ 0 \le l \le d-1, \\
&& p_0^d \mapsto p_0^d+\lambda(1-g^d), \\
&& p_i^d \mapsto p_i^d, \quad \  i \ne 0.
\end{eqnarray*}
defines a coalgebra automorphism of $k\mathcal{A}[d].$ There exists
a coalgebra automorphism $F^d_\lambda: \ k\mathcal{A}
\longrightarrow k\mathcal{A}$ such that its restriction to
$k\mathcal{A}[d]$ is $f_\lambda^d.$
\end{lem}

\subsection{}
We collect in this subsection some useful results of graded Hopf
structures on $k\mathcal{A}.$ The graded Hopf structures are in
one-one correspondence to the left $kG$-module structures on $ke_0,$
in turn to non-zero elements of $k.$ Assume $g.e_0=qe_0$ for some $0
\ne q \in k.$ The corresponding $kG$-Hopf bimodule is $ke_0 \otimes
kG.$ We identify $e_i$ and $e_0 \otimes g^i,$ and in this way we
have a $kG$-Hopf bimodule structure on $k\mathcal{A}_1.$ We denote
the corresponding graded Hopf algebra by $k\mathcal{A}(q).$ The
following lemma give the presentation of $k\mathcal{A}(q)$ by
generators with relations. The proof is routine as Lemma 3.2, so is
omitted.

\begin{lem}
The algebra $k\mathcal{A}(q)$ can be presented via generators with
relations as follows:
\begin{enumerate}
  \item when $q=1,$ generators: $g, \ g^{-1}, \ e_0.$ relations: $gg^{-1}=1=g^{-1}g, \ ge_0=e_0g.$
  \item when $q \ne 1$ is not a root of unity, generators: $g, \ g^{-1}, \ e_0.$
        relations: $gg^{-1}=1=g^{-1}g, \ ge_0=qe_0g.$
  \item when $q \ne 1$ is a root of unity of order $d,$ generators: $g, \ g^{-1}, \ e_0, \
        p_0^d.$ relations: $gg^{-1}=1=g^{-1}g, \ e_0^d=0, \ ge_0=qe_0g, \ gp_0^d=p_0^dg, \ e_0p_0^d=p_0^de_0.$
\end{enumerate}
\end{lem}

\subsection{}
With the algebraic characterization of $k\mathcal{A}(q),$ we can
proceed to the possible preferred deformation. The classification of
Hopf structures on $k\mathcal{A}$ is given in the following.

\begin{thm} Let $H$ be a Hopf structure on $k\mathcal{A}$
with $\operatorname{gr}H \cong k\mathcal{A}(q).$ Then as algebra, it
can be presented by generators and relations as follows:
\begin{enumerate}
  \item if $q=1,$ generators: $g, \ g^{-1}, \ e_0.$ relations: $gg^{-1}=1=g^{-1}g, \
  ge_0g^{-1}=e_0+\lambda(1-g)$ with $\lambda \in \{0,1\}.$
  \item if $q \ne 1$ and is not a root of unity, generators: $g, \ g^{-1}, \ e_0.$ relations: $gg^{-1}=1=g^{-1}g, \ ge_0=qe_0g.$
  In particular, $H$ is isomorphic to $k\mathcal{A}(q).$
  \item if $q \ne 1$ is a root of unity of order $d,$ generators: $g, \ g^{-1}, \ e_0, \ p_0^d.$ relations: $gg^{-1}=1=g^{-1}g, \
  e_0^d=0, \ ge_0=qe_0g, \ e_0p_0^d=p_0^de_0, \ gp_0^d-p_0^dg=\lambda(g-g^{d+1})$ with $\lambda \in k.$
\end{enumerate}
\end{thm}

The idea of the proof is the same as that of the basic cycle case.

\subsection{}
Firstly we consider the case of $q=1.$ Assume that $H$ is a Hopf
algebra on $k\mathcal{A}$ with $\operatorname{gr}H \cong
k\mathcal{A}(1).$ Then as algebra, it is generated by $g, \ g^{-1},
\ e_0$ according to Lemmas 2.1 and 4.2. So in order to get the
defining relations, all we need to do is determine the deformations
of $ge_0g^{-1}=e_0.$ By
\[
\D(ge_0g^{-1})=(ge_0g^{-1}) \otimes 1 + g \otimes (ge_0g^{-1}),
\]
we have $ge_0g^{-1} = e_0+\lambda (1-g)$ for some $\lambda \in k.$
If $\lambda \ne 0,$ then let $E := \frac{e_0}{\lambda},$ we have
$gEg^{-1} = E+(1-g).$ Note that $H$ is generated by $g, \, E,$
therefore through the coalgebra automorphism
\begin{eqnarray*}
F: && k\mathcal{A} \longrightarrow k\mathcal{A} \\
&& g^i \mapsto g^i, \ \ \forall \ i \in \mathbb{Z}, \\
&& e_0 \mapsto \frac{e_0}{\lambda}, \ \ e_i \mapsto e_i, \ \ \forall \ i \ne 0, \\
&& p_i^l \mapsto \lambda^t p_i^l, \ \ \hbox{if $e_0$ appears $t$
   times in $p_i^l,$} \ \forall \ i \in \mathbb{Z}, \ \forall \ l \ge 2
\end{eqnarray*}
we can always reduce the relation in $H$ to the equation $ge_0g^{-1}
= e_0+(1-g)$ when $\lambda \ne 0.$

On the contrary, similar to the argument in Subsection 3.7, it is
not difficult to verify that the relations in Theorem 4.3 (1) are
actually enough to define the algebra structure of $H.$

\subsection{}
Next we consider the case when $q \ne 1$ is not a root of unity.
Assume that $H$ is a Hopf algebra on $k\mathcal{A}$ with
$\operatorname{gr}H \cong k\mathcal{A}(q).$ Again, as algebra, it is
generated by $g, \ g^{-1}, \, e_0$ according to Lemmas 2.1 and 4.2.
So we need to determine the deformation of $ge_0g^{-1}=qe_0$ to get
defining relations for $H.$ Similar to the previous argument, we
have $ge_0g^{-1} = qe_0+\lambda (1-g)$ for some $\lambda \in k.$ Now
let $\widetilde{e_0}=e_0+\frac{\lambda}{1-q},$ then we have
\[g\widetilde{e_0}g^{-1} = q\widetilde{e_0}.\]  By Lemma 4.1, we can
have a coalgebra isomorphism for the coalgebra $k\mathcal{A}$ which
sends $e_0$ to $\widetilde{e_0}.$ Now under the isomorphism, the
original Hopf structure can be transported to a new one with
relation \[ ge_0g^{-1} = qe_0. \] So in this case, the Hopf
structures are graded, and we are done with (2) of the theorem.

\subsection{}
Finally we deal with the case when $q \ne 1$ is a root of unity of
order $d.$ Assume that $H$ is a Hopf algebra on $k\mathcal{A}$ with
$\operatorname{gr}H \cong k\mathcal{Z}(q).$ In this situation, the
Hopf algebra $H$ is generated by $g, \ g^{-1}, \ e_0, \ p_0^d.$ By a
similar argument we have in the first place
\begin{eqnarray*}
&& ge_0=qe_0g, \ e_0^d=\lambda(1-g^d), \ gp_0^dg^{-1}=p_0^d +
\alpha(1-g^d), \\ &&
[e_0,p_0^d]=\frac{\lambda(1-q)}{(d-1)_q!}e_0(1+g^d)+\mu(1-g^{d+1})
\end{eqnarray*}
with some $\lambda, \alpha, \mu \in k.$ By induction we have
\[
e_0^dp_0^d=p_0^de_0^d +
d\frac{\lambda(1-q)}{(d-1)_q!}e_0(1+g^d)+d\mu e_0^{d-1} \ .
\]
Combined with the previous equalities, we have
\[
\lambda d \alpha(g^d-g^{2d})+
d\frac{\lambda^2(1-q)}{(d-1)_q!}(1-g^{2d})+d\mu e_0^{d-1}=0 \ .
\]
It follows from this equality that $\lambda=\mu=0.$ By a similar
argument we can prove the relations are enough defining relations
for $H.$

\subsection{}
To conclude this section, we summarize all the Hopf structures
arising from the linear chain $\mathcal{A}.$ Let
$k\mathcal{A}(1,\lambda)$ denote the Hopf algebra defined by (1) of
Theorem 4.3, and $k\mathcal{A}(d,q,\lambda)$ the Hopf algebra
defined by (3) of Theorem 4.3. The condition of isomorphism is also
given.

\begin{thm}
Let $\mathcal{A}$ be the linear chain and $k\mathcal{A}$ the
associated path coalgebra. All the Hopf algebra structures are given
by $k\mathcal{A}(q)$ with $0 \ne q \in k$ an arbitrary element
(graded), $k\mathcal{A}(1,\lambda)$ and $k\mathcal{A}(d,q,\lambda)$
with $q \ne 1$ a primitive $d$-th root of unity (non-graded). We
have Hopf algebra isomorphism $k\mathcal{A}(q) \cong
k\mathcal{A}(q')$ if and only if $q=q';$ $k\mathcal{A}(1,\lambda)
\cong k\mathcal{A}(1,\lambda')$ if and only if $\lambda =\lambda';$
and $k\mathcal{A}(d,q,\lambda) \cong k\mathcal{A}(d',q',\lambda')$
if and only if $d=d', \ q=q',\ \lambda = \lambda'.$
\end{thm}

\section{Applications}
In this section, we give some direct applications of the obtained
classification results to bialgebras of type one of Nichols
\cite{n}, and simple-pointed Hopf algebras of Radford \cite{rad}.

\subsection{}
Recall that the bialgebras of type one in the sense of Nichols are
pointed Hopf algebras that are generated as algebras by group-like
and skew-primitive elements. In the quiver terminology, such Hopf
algebras live in Hopf quivers and as algebras are generated by
vertices and arrows. We are going to investigate all the possible
bialgebras of type one living in minimal Hopf quivers. Note that not
all pointed Hopf algebras are bialgebras of type one. Later on we
can see that in general quiver Hopf algebras are not so.

The case of loop quiver is trivial. The quiver Hopf algebra is
generated by the only vertex and arrow, hence is bialgebra of type
one. For the cases of basic $n$-cycles ($n \ge 2$) and the linear
chain, things turn out to be very different. The idea of classifying
bialgebras of type one is similar to those of quiver Hopf algebras.
First we classify the graded ones, and then determine all the
possible deformations.

\subsection{}
We deal with the basic cycle case first. Keep the notations of
Section 3. The graded bialgebras of type one are the graded sub-Hopf
algebras on Hopf quivers generated by vertices and arrows, so the
classification of such algebras can be obtained as a direct
consequence of Lemma 3.2.

\begin{lem} Let $B\mathcal{Z}(q)$ denote the sub-Hopf algebra of $k\mathcal{Z}(q)$
generated by vertices and arrows.
\begin{enumerate}
  \item If $q=1,$ then $B\mathcal{Z}(q) \cong k\mathcal{Z }(1).$
  \item If $\operatorname{ord}(q)=d > 1,$ then $B\mathcal{Z}(q)$ can
        be presented by generators $g, \ a_0$ with relations $g^n=1, \ a_0^d=0, \ ga_0=qa_0g.$
\end{enumerate}
\end{lem}

We remark that when $q$ is a non-trivial root of unity, the
bialgebra of type one $B\mathcal{Z}(q)$ is a very interesting Hopf
algebra. In particular, when $\operatorname{ord}(q)=n,$ it is the
well-known Taft algebra \cite{taft}. It also appears as the Borel
subalgebra of Lusztig's small quantum $sl_2$ \cite{lusztig}. For
general $\operatorname{ord}(q)=d,$ the algebra $B\mathcal{Z}(q)$ is
a generalization of the Taft algebra.

It follows directly from Lemmas 3.2 and 5.1 that the only finite
dimensional subcoalgebras of $k\mathcal{Z}$ which admit Hopf algebra
structures are $k\mathcal{Z}[d]$ with $d=\operatorname{ord}(q),$
where $q$ is a non-trivial $n$-th root of unity. Hence $d$ is a
factor of $n.$ This is the Theorem 3.1 of \cite{chyz}, which plays
an important role in classifying the monomial Hopf algebras. The
argument in this paper simplifies the old one.

For the non-graded bialgebras of type one living in the Hopf quiver
$\mathcal{Z},$ it suffices to determine the preferred deformations
of $B\mathcal{Z}(q).$ This actually was done in \cite{chyz}, though
not in terms of deformation. It turns out that these are all the
connected monomial Hopf algebras. For completeness, we include the
result here.

\begin{thm} \rm{(\cite{chyz}, Theorem 3.6)}
All the possible preferred deformations of $B\mathcal{Z}(q)$ can be
presented by generators $g, \ a$ with relations \[ g^n=1, \
a^d=\mu(1-g^d), \ ga=qag, \] where $\mu \in \{ 0, \ 1 \}.$
\end{thm}

\subsection{}
Now we consider the case of linear chain. Keep the notations of
Section 4. First by Lemma 4.2, we can classify the graded bialgebras
of type one.

\begin{lem}
Let $B\mathcal{A}(q)$ denote the sub-Hopf algebra of
$k\mathcal{A}(q)$ generated by vertices and arrows.
\begin{enumerate}
  \item If $q$ is not a non-trivial root of unity, then $B\mathcal{A}(q) \cong k\mathcal{A }(q).$
  \item If $q$ is a root of unity with order $\operatorname{ord}(q)=d > 1,$ then $B\mathcal{A}(q)$ can
        be presented by generators $g, \ g^{-1}, \ e_0$ with relations $gg^{-1}=1=g^{-1}g, \ e_0^d=0, \ ge_0=qe_0g.$
\end{enumerate}
\end{lem}

Next we consider the possible deformation of $B\mathcal{A}(q).$ For
the case of $q$ being not a non-trivial root of unity, it is done in
Theorem 4.3. When $q$ is a root of unity with order
$\operatorname{ord}(q)=d > 1,$ it suffices to deform the relations
$e_0^d=0, \ ge_0=qe_0g.$ By the same argument as those in
Subsections 3.5 and 4.6, we can always preserve the relation
$ge_0=qe_0g,$ while deform $e_0^d=0$ to $e_0^d=\lambda(1-g^d).$ In
this situation, as coalgebra $B\mathcal{A}(q)$ is identical to the
subcoalgebra $k\mathcal{A}[d-1]$ of the path coalgebra
$k\mathcal{A},$ namely the subcoalgebra spanned by paths of length
strictly less than $d.$ We record the results as follows.

\begin{thm}
If $H$ is a bialgebra of type one with $Q(H)=\mathcal{A},$ then $H$
is one of the following:
\begin{enumerate}
  \item $k\mathcal{A}(q),$ where $q$ is not a non-trivial root of
        unity.
  \item H can be presented by generators $g, \ g^{-1}, \ e$ with
        relations  $gg^{-1}=1=g^{-1}g, \ ge=qeg, \
        e^d=\mu(1-g^d),$ where $q$ is a root of unity of order
        $d > 1$ and $\mu \in \ \{0,1\}.$
\end{enumerate}
\end{thm}

We remark that when $q$ is not a root of unity, the Hopf algebra
$k\mathcal{A}(q)$ is the well-known Borel subalgebra of the quantum
group $U_\nu(sl_2),$ here $\nu=\sqrt{q};$ while $q$ is a non-trivial
root of unity, the Hopf algebra in (2), denoted by $B\mathcal{A}(q,
\mu),$ is closely related to the (Borel subalgebra of) De
Concini-Kac quantum group $U_\nu(sl_2)$ at roots of unity
\cite{dck}.

\subsection{}
By comparing the previous classification of bialgebras of type one
with the classification of quiver Hopf algebras, it is clear that
there is no hope to extend the Gabriel type theorem of Van Oystaeyen
and Zhang to non-graded pointed Hopf algebras. For example, when $q$
is a non-trivial root of unity, the Hopf algebras $B\mathcal{Z}(q)$
and $B\mathcal{A}(q)$ do have non-trivial deformations, while the
quiver Hopf algebras $k\mathcal{Z}(q)$ and $k\mathcal{A}(q)$ do not
in general. In other words, these Hopf algebras living in a proper
subcoalgebra of $k\mathcal{Z}$ and $k\mathcal{A}$ can not be
extended to the whole path coalgebras, hence they are not sub-Hopf
algebras of any Hopf algebra structures on the path coalgebras of
the corresponding Hopf quivers.

\subsection{}
Now we apply the obtained results to simple-pointed Hopf algebras. A
Hopf algebra $H$ is said to be simple-pointed, if it is pointed, not
cocommutative, and if $L$ is a proper sub-Hopf algebra of $H,$ then
$L \subseteq kG(H).$ Very naturally the ``simpleness" of such Hopf
algebras can be visualized by their corresponding Hopf quivers. We
remark that the definition of simple-pointed Hopf algebras adopted
here is from \cite{z} which includes infinite-dimensional situation.
In there the complete list of classification was obtained by
different methods from ours.

\begin{thm}
A Hopf algebra $H$ is simple-pointed if and only if its graded
version $\gr H$ is simple-pointed, if and only if it is a bialgebra
of type one living in either the Hopf quiver $\mathcal{Z}$ or
$\mathcal{A}.$ Hence it is one of the following: $k\mathcal{Z}(1),$
$B\mathcal{Z}(q,\mu)$ \rm{($q$ is a root of unity of order $> 1,$
$\mu \in \{0,1\}$),} $k\mathcal{A}(q)$ \rm{($q$ is not a non-trivial
root of unity),} $k\mathcal{A}(1,1),$ $B\mathcal{A}(q, \mu)$
\rm{($q$ is a root of unity of order $> 1,$ $\mu \in \{0,1\}$).}
\end{thm}

\begin{proof}
Assume that $H$ is simple-pointed, then the Hopf quiver $Q(H)$ must
be connected. It is not hard to deduce from the definition of
simple-pointed Hopf algebras that there is exactly one arrow going
from the unit of the group $G(H)$ to some non-unit element. So by
the definition of Hopf quiver it follows at once that $G(H)$ must be
a cyclic group and $Q(H)$ must be either $\mathcal{Z}$ or
$\mathcal{A}.$ Now the theorem follows directly from Lemma 5.1,
Theorem 5.2, Lemma 5.3, and Theorem 5.4.
\end{proof}

\vskip 10pt

\noindent {\bf Acknowledgements:} The authors were supported
partially by the NSF of China (Grant No. 10501041, 10526037,
10601052). Part of the work was done while the first author was
visiting the Abdus Salam International Centre for Theoretical
Physics (ICTP). He expresses his sincere gratitude to the ICTP for
its support.


\begin{thebibliography}{999}

\bibitem{as1}
N. Andruskiewitsch, H.-J. Schneider, Lifting of quantum linear
spaces and pointed Hopf algebras of order $p^3,$  J. Algebra 209
(1998) 658-691.

\bibitem{as2}
N. Andruskiewitsch, H.-J. Schneider, On the classification of
finite-dimensional pointed Hopf algebras, Ann. Math., to appear.
math.QA/0502157.

\bibitem{ass}
I. Assem, D. Simson, A. Skowronski, Elements of the Representation
Theory of Associative Algebras. Vol. 1. Techniques of Representation
Theory. London Mathematical Society Student Texts, 65. Cambridge
University Press, Cambridge, 2006.

\bibitem{b}
G. Bergman, The diamond lemma for ring theory, Adv. Math. 29 (1978)
178-218.

\bibitem{chyz}
X.-W. Chen, H.-L. Huang, Y. Ye, P. Zhang, Monomial Hopf algebras, J.
Algebra 275 (2004) 212-232.

\bibitem{cm}
W. Chin, S. Montgomery, Basic coalgebras, Modular interfaces
(Riverside, CA, 1995), 41-47, AMS/IP Stud. Adv. Math. 4, Amer. Math.
Soc., Providence, RI, 1997.

\bibitem{c1}
C. Cibils, A quiver quantum group, Comm. Math. Phys. 157 (1993)
459-477.

\bibitem{cr1}
C. Cibils, M. Rosso, Alg\`{e}bres des chemins quantiques, Adv. Math.
125 (1997) 171-199.

\bibitem{cr2}
C. Cibils, M. Rosso, Hopf quivers, J. Algebra 254 (2002) 241-251.

\bibitem{dck}
C. De Concini, V.G. Kac, Representations of quantum groups at roots
of 1, in ``Operator Algebras, Unitary Representations, Enveloping
Algebras, and Invariant Theory", ed. A. Connes etc (2000),
Birkh\"auser, 471-506.

\bibitem{gab}
P. Gabriel, Unzerlegbare Darstellungen I, Manuscripta Math. 6 (1972)
71-103.

\bibitem{grk}
P. Gabriel, A.V. Roiter, Representations of finite-dimensional
algebras. Translated from the Russian. With a chapter by B. Keller.
Reprint of the 1992 English translation. Springer-Verlag, Berlin,
1997.

\bibitem{g-s}
M. Gerstenhaber, S.D. Schack, Bialgebra cohomology, deformations,
and quantum groups, Proc. Nat. Acad. Sci. 87 (1990) 478-481.

\bibitem{g}
E.L. Green, Constructing quantum groups and Hopf algebras from
coverings, J. Algebra 176 (1995) 12-33.

\bibitem{gs}
E.L. Green, \O. Solberg, Basic Hopf algebras and quantum groups,
Math. Z. 229 (1998) 45-76.

\bibitem{kassel}
C. Kassel, Quantum groups, Graduate Texts in Mathematics, 155.
Springer-Verlag, New York, 1995.

\bibitem{lusztig}
G. Lusztig, Finite dimensional Hopf algebras arising from quantized
universal enveloping algebras, Journal of the AMS 3 (1990) 257-296.

\bibitem{mont1}
S. Montgomery, Hopf Algebras and Their Actions on Rings, CBMS
Regional Conf. Series in Math. 82, Amer. Math. Soc., Providence, RI,
1993.

\bibitem{mont2}
S. Montgomery, Indecomposable coalgebras, simple comodules and
pointed Hopf algebras, Proc. of the Amer. Math. Soc. 123 (1995)
2343-2351.

\bibitem{n}
W.D. Nichols, Bialgebras of type one, Communications in Algebra
6(15) (1978) 1521-1552.

\bibitem{rad}
D. Radford, Finite-dimensional simple-pointed Hopf algebras, J.
Algebra 211 (1999) 686-710.

\bibitem{rosso}
M. Rosso, Quantum groups and quantum shuffles, Invent. Math. 133
(1998) 399-416.

\bibitem{sw}
M. Sweedler, Hopf algebras, W. A. Benjamin, Inc., New York, 1969.

\bibitem{taft}E.J. Taft, The order of the antipode of finite
dimensional Hopf algebras, Prc. Nat. Acad. Sci. USA 68 (1971)
2631-2633.

\bibitem{voz}
F. Van Oystaeyen, P. Zhang, Quiver Hopf algebras, J. Algebra 280(2)
(2004) 577-589.

\bibitem{z}
P. Zhang, Hopf algebras on Schurian quivers, Comm. Algebra 34(11)
(2006) 4065-4082.

\end{thebibliography}
\end{document}